\documentclass[12pt]{amsart}
\usepackage{amssymb, amscd, amsmath, amsfonts, amsthm, mathrsfs}
\usepackage[T1]{fontenc}
\usepackage{mathrsfs}
\usepackage{anysize}
\usepackage[margin=2cm]{geometry}

\newcommand{\x}{\textbf}
\newcommand{\q}{\quad}

\newcommand{\mbb}{\mathbb}
\newcommand{\p}{\prime}
\theoremstyle{plain}
\newtheorem{thm}{Theorem}[section]
\newtheorem{lem}[thm]{Lemma}
\newtheorem{prop}[thm]{Proposition}
\newtheorem{cor}[thm]{Corollary}
\newtheorem{defi}[thm]{Definition}
\newcommand{\thmref}[1]{Theorem~\ref{#1}}
\newcommand{\lemref}[1]{Lemma~\ref{#1}}

\newcommand{\propref}[1]{Proposition~\ref{#1}}
\newcommand{\corref}[1]{Corollary~\ref{#1}}
\newtheorem{rmk}[thm]{Remark}

\begin{document}

\title{Non-vanishing of Jacobi Poincar\'{e} series}

\author{SOUMYA DAS}
\address{Harish Chandra Research Institute\\ 
         Chhatnag Road\\  
         Jhusi Allahabad 211019, India.}
\email{somu@hri.res.in, soumya.u2k@gmail.com}
%\date{\today}

\begin{abstract}
We prove that under suitable conditions, the Jacobi Poincar\'{e} series of exponential type of integer weight and matrix index does not vanish identically. For classical Jacobi forms, we construct a basis consisting of the ``first'' few Poincar\'{e} series and also give conditions both dependent and independent of the weight, which ensures non-vanishing of classical Jacobi Poincar\'{e} series. Equality of certain Kloosterman-type sums is proved. Also, a result on the non-vanishing of Jacobi Poincar\'{e} series is obtained when an odd prime divides the index.  
\end{abstract}
\maketitle

\section{Introduction}
\numberwithin{equation}{section}
In \cite{rankin} R. A. Rankin has proved that the $m$-\textit{th} Poincar\'{e} series $P^{k}_{m}$
of weight $k$, where $k,m$ are positive integers, for the full modular group $SL(2,\mbb{Z})$ is not identically 
zero for sufficiently large $k$ and finitely many $m$ depending on $k$. C. J. Mozzochi extended Rankin's result to integral weight modular forms for congruence subgroups in \cite{mozzochi}. 

In this paper we prove similar results for higher degree Jacobi Poincar\'{e} series defined on the full Jacobi group $\Gamma_{g}^{J} = SL(2,\mbb{Z}) \ltimes ( \mbb{Z}^{g} \times \mbb{Z}^{g}) $, where $g$ is a positive integer and is referred to as the degree of the Jacobi group. The Jacobi group operates on $\mathcal{H} \times \mbb{C}^{g}$ and also on functions $ \phi \colon \mathcal{H} \times \mbb{C}^{g} \rightarrow \mbb{C}$. We denote the latter action by $\mid_{k,m}$. (See section~\ref{preliminaries} for the definitions.)

Let $k,g \in \mbb{Z}$, $m$ a symmetric, positive-definite, half-integral $(g \times g)$ matrix. The vector space of Jacobi cusp forms of weight $k$, index $m$ and degree $g$, denoted by $J_{k,m,g}^{cusp}$ is defined to be the space of holomorphic functions $\phi \colon  \mathcal{H} \times \mbb{C}^{g} \rightarrow \mbb{C}$ satisfying $\phi|_{k,m}\gamma = \phi$ (where $\gamma \in \Gamma_{g}^{J}$) and having a Fourier expansion \[ \phi(\tau,z) = \underset{n \in \mbb{N}, r \in \mbb{Z}^{g} , 4n > m^{-1}[r^{t}]} \sum c_{\phi}(n,r) e(n\tau+rz) \]
If $g=1$, we denote $J_{k,m,1}^{cusp}$ by $J_{k,m}^{cusp}$.

For $n \in \mbb{N}$, $r \in \mbb{Z}^{g}$ with $4n > m^{-1}[r^{t}]$, let $P_{n,r}^{k,m}$ be the $(n,r)$-\textit{th} Poincar\'{e} series of weight $k$ and index $m$ (of exponential type) defined for $k > g + 2$ as in \cite{kohnen} (see Section~\ref{preliminaries} for definition). It is well-known that the Poincar\'{e} series $P_{n,r}^{k,m}$ ($n \in \mbb{Z}, r \in \mbb{Z}^{g}$) span $J_{k,m,g}^{cusp}$. It is then natural to ask when such a Poincar\'{e} series vanish identically or when it does not. We prove the following theorem, which gives a partial answer to the above question.

Let  $D  = \det {\left( \begin{smallmatrix}
2n & r \\
r^{t} & 2m  \end{smallmatrix}\right) }$ and define $k^{\p} := k - g/2 -1$.

\begin{thm} \label{order1+}
Let $k$ be even when  $2r \equiv 0 \pmod{\mbb{Z}^{g} \cdot 2m }$. Then there exist an integer $k_{0}$ and a constant $B > 3 \log{2}$ such that for all $k \geq k_{0}$ (depending only on $g$), the Jacobi Poincar\'{e} series $P_{n,r}^{k,m}$ does not vanish identically for \[ k^{\p} \leq \frac{\pi D}{\det{(2m)}} \leq {k^{\p}}^{1+\alpha(g)}  \exp \left\{- \frac{  B \log{k^{\p}}}{\log{ \log{k^{\p}}}} \right\} ,\]
where $\alpha(g) = \left\{ \begin{array}{ccc} 
\frac{2}{3(g+2)} & \mbox{ if } & 1 \leq g \leq 4, \\
\frac{2}{3g} & \mbox{ if }  & g \geq 5. 
\end{array} \right.$ 
\end{thm}

We construct a basis of $ J_{k,m}^{cusp}$ consisting of the ``first'' $\dim{ J_{k,m}^{cusp}} $ Poincar\'{e} series (see \thmref{classicalbasis} in section~\ref{g=1}). We also give conditions for non-vanishing of Poincar\'{e} series independent of the weight for classical Jacobi forms ($g=1$). 

Define $M(x) := \exp \left\{ \frac{  B_{1} \log{x}} {\log{ \log{2x}}} \right\}$ ($x \geq 2, \, B_{1} > \log{2}$) as in \cite{rankin}.

\begin{thm} \label{g=1,indepofk}
Let $g=1$. For $D > \frac {m}{\pi}$, we have $P^{k,m}_{D,r} \not \equiv 0$ for \[  M \left( \frac{\pi D}{m} \right) \, \sigma_{0}(D) \, D < \frac{m^{\frac{8}{7} }} { \lambda} , \]  where $\lambda = ( 2 \sqrt{2} \pi^{  \frac{5}{3}} A)^{ \frac{3}{2}}  $, $ A = \frac{1} { \pi} \left( \frac{2} {6^{\frac{2}{3} }} + \frac{54}{2^{\frac{5}{6}}} + \frac{16}{2^{\frac{3}{4}}} \right)$ 
and $\sigma_{0} (D) = \underset{d \mid D} \sum 1$.
\end{thm}

Finally following \cite{rankin}, we give conditional statements on the non-vanishing of Jacobi Poincar\'{e} series, based on the relation of $g$-dimensional Kloosterman sums with corresponding $1$-dimensional sums and identities involving them.

\begin{thm}\label{consecutive}
Let $p$ be an odd prime, $\mu \in \mbb{N}$. Suppose \q $ p|(m,r), \,  p \nmid n.$ If $P_{p^{\mu}n,p^{\mu}r}^{k,p^{\mu}m} \not\equiv 0$  then 
\begin{equation}\label{consecutivekloo} \mbox{either } \, P_{np^{\mu-1},rp^{\mu-1}}^{k,mp^{\mu-1}} \not\equiv 0 \q \q or \q P_{np^{2\mu},rp^{2\mu}}^{k,p^{2\mu}m} \not\equiv 0 \q and \q P_{n,rp^{\mu}}^{k,p^{2\mu}m} \not\equiv 0 \end{equation}
\end{thm}
(Here $p | m$ means $p$ divides every entry of $m$; since $2m$ is a $(g \times g)$ matrix with integer entries and $p$ is odd, this makes sense.)

\begin{rmk}
\begin{enumerate}
\item
In Section~\ref{proofs} we first prove the trivial case where the Poincar\'{e} series $P_{n,r}^{k,m}$ does not vanish when the ratio $ \frac{\pi D}{\det{(2 m)}} \left( \frac{ D}{\det{(2 m)}} = 2n - 2 m^{-1}[\frac{1}{2}r^{t}] \right)$ by which we measure the non-vanishing of Jacobi Poincar\'{e} series, is $O(k) $, but with explicit range of the weight $k$ where this is valid. This follows from \propref{poincarenon-vanish} for arbitrary $g$ and also from \thmref{classicalbasis} in the case $g=1$ (recall that $\dim{ J_{k,m,1}^{cusp} }\sim O \left( \frac{k (m+1)}{12} \right)$ ). \\

\item
\thmref{order1+} therefore improves the trivial case mentioned in the previous remark. However, achieving the ``order of $k^{2 - \epsilon}$ ($\epsilon >0$)'' as in \cite{rankin} in the case of Jacobi Poincar\'{e} series using Rankin's methods seems difficult mainly because of the presence of the factor $(c,D)$ instead of $(c,D)^{\frac{1}{2}}$ in the estimate of Kloosterman sums of degree $g$ (even for small $g$, see section \ref{proofs}). \\

\item
The condition that $k$ be even when $2r \equiv 0 \pmod{\mbb{Z}^{g} \cdot 2m }$ in \thmref{order1+} is necessary, as the $(n,r)$-\textit{th} Poincar\'{e} series vanish when $k$ is odd and $2r  \equiv 0 \pmod{\mbb{Z}^{g} \cdot 2m }$. The restriction $k^{\p} \leq \frac{\pi D} {\det{(2m)}}$ in \thmref{order1+} is natural since we know the result in the complement (see \propref{poincarenon-vanish}). Same is true for the condition $D > \frac {m} {\pi}$ in \thmref{g=1,indepofk}. 

\end{enumerate}
\end{rmk}

\flushleft \x{Acknowledgements.} The author wishes to thank Prof. B. Ramakrishnan for going through the manuscript and for his support and encouragement.

\section{Notations and Preliminaries} \label{preliminaries}

The Jacobi group $\Gamma_{g}^{J}$ operates on $\mathcal{H} \times \mbb{C}^{g}$ in the usual way by 
\[ \left( \left( \begin{matrix}
a & b \\
c & d  \end{matrix}\right),(\lambda,\mu) \right)\circ(\tau,z) = \left(  \frac{a\tau+b}{c\tau+d} , (c\tau+d)^{-1}(z+\lambda \tau+ \mu) \right). \]

Let $k \in \mbb{Z}$, $m$ a symmetric, positive-definite, half-integral $(g \times g)$ matrix. Then we have the action of $\Gamma_{g}^{J}$ on functions $\phi \colon  \mathcal{H} \times \mbb{C}^{g} \rightarrow \mbb{C}$ given by :
\[ \phi|_{k,m}\gamma(\tau,z) := (c\tau+d)^{-k} e\left(-c(c\tau+d)^{-1}m[z+\lambda \tau + \mu] + m[\lambda] + 2\lambda^{t}mz\right)
\phi \left(\gamma \circ (\tau,z) \right). \]
(Here $A[B] = B^{t}AB$ for matrices $A,B$ of appropriate sizes, $B^{t}$ is the transpose of the matrix $B$, $e(z):= e^{2 \pi i z}$ , $\mathcal{H}$ is the upper half plane.)

The vector space of Jacobi cusp forms of weight $k$, index $m$ and degree $g$, denoted by $J_{k,m,g}^{cusp}$ is defined to be the space of holomorphic functions $\phi \colon  \mathcal{H} \times \mbb{C}^{g} \rightarrow \mbb{C}$ satisfying $\phi|_{k,m}\gamma = \phi$ and having a Fourier expansion \[ \phi(\tau,z) = \underset{n \in \mbb{N}, r \in \mbb{Z}^{g} , 4n > m^{-1}[r^{t}]} \sum c_{\phi}(n,r) e(n\tau+rz) \]

For $n \in \mbb{N}$, $r \in \mbb{Z}^{g}$ with $4n > m^{-1}[r^{t}]$, let $P_{n,r}^{k,m}$ be the $(n,r)$-\textit{th} Poincar\'{e} series of weight $k$ and index $m$ (of exponential type) defined for $k > g + 2$ by
\[ P_{n,r}^{k,m}(\tau , z) := \underset{\gamma \in \Gamma_{g,\infty}^{J} \backslash  \Gamma_{g}^{J}}\sum  e(n \tau + rz)|_{k,m} \gamma(\tau,z) \q \q \tau \in \mathcal{H} ,\, z \in \mbb{C}^{g}, \]
where $\Gamma_{g,\infty}^{J} := \left\{ \left( \left( \begin{matrix}
1 & n \\
0 & 1  \end{matrix}\right), (0, \mu) \right) \mid n \in \mbb{Z}, \mu \in \mbb{Z}^{g} \right\}.$

It is well known that $J_{k,m,g}^{cusp}$ is finite dimensional and the family of Poincar\'{e} series $P_{n,r}^{k,m}$ ($n \in \mbb{N}$, $r \in \mbb{Z}^{g}$) form a basis of $J_{k,m,g}^{cusp}$. In \cite[Lemma 1]{kohnen}, S. B$\ddot{\mbox{o}}$cherer and W. Kohnen obtained the Fourier expansion of $P_{n,r}^{k,m} \colon $ 

\begin{prop}\label{poincarefourier}
(1) The function $P_{n,r}^{k,m}$ is in $J_{k,m,g}^{cusp}$. The Fourier expansion of the Poincar\'{e} series is given by

\begin{equation*}
P_{n,r}^{k,m}(\tau ,z) = \underset{n^{\p} \in \mbb{N}, r^{\p} \in \mbb{Z}^{g} , 4n^{\p} > m^{-1}[{r^{\prime}}^{t}]} \sum 
c_{n,r}^{k,m}(n^{\p},r^{\p}) e(n^{\p} \tau + r^{\p}z)  , 
\end{equation*} where

\begin{eqnarray}
\begin{split}
c_{n,r}^{k,m}(n^{\p},r^{\p}) & = \delta_{m}(n,r,n^{\p},r^{\p}) + (-1)^{k}\delta_{m}(n,r,n^{\p},-r^{\p}) +  2 \pi i^{k}\, det(2m)^{-1/2} \cdot (D^{\p}/D)^{k/2 - g/4 - 1/2}\\
& \times \underset{c \geq 1}\sum \left( H_{m,c}(n,r,n^{\p},r^{\p})+ (-1)^{k} H_{m,c}(n,r,n^{\p},-r^{\p}) \right) J_{k-g/2 -1}\left( \frac{2 \pi \sqrt{DD^{\p}}}{det(2m)\cdot c} \right) \\ 
\end{split}
\end{eqnarray}
where $D^{\p}  = \det {\left( \begin{smallmatrix}
2n^{\p} & r^{\p} \\
{r^{\p}}^{t} & 2m  \end{smallmatrix}\right) }$ \q , \q $\delta_{m}(n,r,n^{\p},r^{\p}): = \left\{ \begin{array}{cc} 1 & \mbox{ if } D = D^{\p} , r \equiv r^{\p} \pmod{\mbb{Z}^{g} \cdot 2m } , \\
0 & \mbox{ otherwise } , \end{array} \right. $\\
and \q $H_{m,c}(n,r,n^{\p},r^{\p}) : = c^{-g/2 -1} \, \underset{x(c),y(c)^{*}}\sum e_{c}\left(( m[x]+ rx +n)\bar{y} + n^{\p}y + r^{\p}x \right) e_{2c}(r^{\p}m^{-1}r^{t})$ \\
where in the summation $x$ (resp. $y$) run over a complete set of representatives for $\mbb{Z}^{(g,1)}/c \mbb{Z}^{(g,1)}$ (resp. $(\mbb{Z}/c \mbb{Z})^{*})$, $\bar{y}$ denotes an inverse of \, $y \pmod{c}, \q e_{c}(a) := e^{2 \pi i a/c} \, (a \in \mbb{Z}) ,$ and $J_{k - g/2 -1}$ denotes the Bessel function of order $k - g/2 -1$. 

(2) $\langle \phi , P_{n,r}^{k,m} \rangle  = \lambda_{k,m,D}\,  c_{\phi}(n,r)$, \,  where $c_{\phi}(n,r)$  denotes the $(n,r)$-\textit{th} Fourier coefficient of $\phi$ and \[ \lambda_{k,m,D} = 2^{(g-1)(k-g/2-1)-g} \cdot \Gamma(k-g/2-1) \cdot \pi^{-k+g/2+1} \cdot (det \,m)^{k-(g+3)/2} \cdot D^{-k+g/2+1}\]   (here $ \langle \, , \, \rangle $ is the Petersson inner product on $J_{k,m,g}^{cusp}$ ).
\end{prop}

From \propref{poincarefourier}, we conclude that the Poincar\'{e} series $P_{n,r}^{k,m}$ is non-zero if and only if it's \textit{(n,r)-th} Fourier coefficient $c_{n,r}^{k,m}$ is positive. So, it is enough to prove $c_{n,r}^{k,m}$ is non-zero. For convenience of notation we will drop the $(k,m)$ in the calculations.

\begin{lem}\label{trivialweight}
The Poincar\'{e} series $ P_{n,r}^{k,m}$ vanishes if $k$ is odd and $2r \equiv 0 \pmod{\mbb{Z}^{g} \cdot 2m }$.
\end{lem}

\begin{proof}
In fact the \textit{(n,r)-th} coefficient $c(n,r)$ of a general Jacobi form of degree $g$ is zero if $k$ is odd when $2r \equiv 0 \pmod{\mbb{Z}^{g} \cdot 2m }$. This is an easy consequence of the transformation property of Jacobi forms. See for instance \cite{zagier} for $g=1$.
\end{proof}

From the Fourier expansion of $P_{n,r}^{k,m}$ we see that in order to prove that it is non-zero, it is enough to prove $|S(n,r)| < \frac{1}{2\pi}$ (noting that $2m$ is a positive-definite matrix with integer entries, hence $ \det{(2m)} \geq 1$), where 
\begin{equation} S(n,r):= \det{(2m)} ^{-1/2} \, \underset{c \geq 1}\sum \left( H_{m,c}(n,r,n,r)+ (-1)^{k} H_{m,c}(n,r,n,-r) \right) J_{k-g/2 -1}\left( \frac{2 \pi D}{det(2m)\cdot c} \right) 
\end{equation}

We will need the following estimates. (See \cite{bessel},\cite{watson} and \cite{kohnen} respectively for details):
\begin{eqnarray}
&\mbox{(i)} \q & |J_{\nu}(x)| \leq \, \min{ \left\{ 1, \frac{1}{\Gamma(\nu +1)} \left(\frac{x}{2}\right)^{\nu}  \right\} } \mbox{ for } x > 0 \mbox{ and } \nu \geq 2. \label{besselestimate}\\
&\mbox{(ii)} \q & |H_{m,c}(n,r,n,\pm r)| \leq \, 2^{\omega(c)} c^{g/2 -1} (D,c), \label{kloostermanestimate}\\
& \mbox{ where }  &\omega(c) \mbox{ is the number of distinct prime divisors of } c,\, (D,c) = gcd(D,c).\nonumber 
\end{eqnarray}

\section{Proof of \thmref{order1+}} \label{proofs}

\subsection{} In this section we first obtain the following proposition which follows easily from trivial estimates of Bessel functions.
\begin{prop}\label{poincarenon-vanish}
 
\begin{enumerate}
\item
Let $k$ be even when  $2r \equiv 0 \pmod{\mbb{Z}^{g} \cdot 2m }$. Then there exists an integer $k_{0}$ such that the \textit{(n,r)-th} Poincar\'{e} series $P_{n,r}^{k,m}$ does not vanish identically  \[ \mbox{ for } \, k  \geq k_{0}\,   \mbox{ and } (n,r) \in \mbb{N} \times \mbb{Z}^{g}  \mbox{ with }  D \leq \frac{k^{\p}}{ \pi e} \cdot \det{(2m)} \] If $k > g+3$, then one can take $k_{0} = max \left( g+4 , [\frac{g}{2}] + 69 \right).$ 

\item
For all \, $ D < \frac{1}{\pi} \, \det{(2m)} $ \, the Poincar\'{e} series $P_{n,r}^{k,m}$ does not vanish identically, whenever the condition is non-void and $k > g+3$ if $g \geq 2$ and $ k > 5$ if $g=1$. 
\end{enumerate}
\end{prop}

\begin{lem} \label{non-void}
The condition in \propref{poincarenon-vanish}$(2)$ is non-void for $n < \frac{1}{6} + \frac{(2m-3)^2}{144m}$ when $g=1$.

\begin{proof} Suppose that $D <  \frac{1}{\pi}(2m) < \frac{1}{3} (2m).$ We have   $2m(2n - \frac{1}{3}) < r^{2} < 4mn.$  Noticing that there is a square in the interval $[x,y]$ , ($x,y \in \mbb{R}^{+}$)  when \, $2 \sqrt{x} + 1 < y -x ,$ \, we need to have,  \[ 2\sqrt{2m(2n-\frac{1}{3})} + 1 < \frac{2m}{3} , \, \mbox{ or } \, n < \frac{1}{6} + \frac{(2m-3)^2}{144m}.\] So, in the case $g=1$ , the Poincar\'{e} series $P_{k,m}^{n,r}$  does not vanish identically when $k>4$ and $n$ satisfies the condition of the lemma.
\end{proof}
\end{lem}

\begin{proof}[\x{Proof of \propref{poincarenon-vanish}}]
In a straightforward manner, using estimates~(\ref{besselestimate}) and~(\ref{kloostermanestimate}), we get  
$ |S(n,r)| \leq  \frac{2}{\Gamma(k^{\p} +1)} \left(\frac{S}{2} \right)^{k^{\p}} \underset{c} \sum \frac{2^{\omega(c)}}{c^{k-g-1}}$, where $S := \frac{2 \pi D}{det(2m)}$. Using this, and \corref{stirling}, the proof follows. We omit the details.
\end{proof}

\begin{cor}\label{stirling}
If $k > g+3$, then one can take $k_{0} = max \left( g+4 , [\frac{g}{2}] + 69 \right)$ in \propref{poincarenon-vanish}.\\  (  \mbox{[.]} denotes the greatest integer function)  
\end{cor}

\begin{proof}
Examining the proof of \propref{poincarenon-vanish}, we see that when $k \geq g+4$ the series 
\[ \underset{c} \sum \frac{2^{\omega(c)}}{c^{k-g-1}} < \zeta(2) = \frac{\pi^{2}}{6}, \]
using the trivial bound $2^{\omega(c)} \leq c$. The rest follows by using Stirling's formula :
\[ n! = \sqrt{2 \pi n} \left( \frac{n}{e} \right)^{n} e^{\lambda_{n}}, \, \mbox{ where } \, \frac{1}{12n +1} < \lambda_{n} < \frac{1}{12 n} , \mbox{ for } n \in \mbb{N}. \]
\end{proof}

\subsection{Poincar\'{e} series for small weights} 
For $ \mathrm{Re}(s) > \frac{1}{2}(\frac{g}{2}-k+2)$, using the 'Hecke trick', the Jacobi Poincar\'{e} series is defined as in \cite{bringmann2}
\[ P_{n,r;s}^{k,m}(\tau,z) = \underset{\gamma \in \Gamma_{g,\infty}^{J} \backslash  \Gamma_{g}^{J}}\sum  \left( \frac{v}{|c\tau+d|^{2}}\right)^{s}e(n \tau + rz)|_{k,m} \gamma(\tau,z) \q  \mbox{ where } \tau=u+iv \in \mathcal{H} ,\, z \in \mbb{C}^{g}, \, s \in \mbb{C}. \]
Then for $k>\frac{g}{2}+2$, \, $P_{n,r;0}^{k,m} \in J_{k,m}^{cusp}$ and has the same Fourier properties as $P_{n,r}^{k,m}$. We also consider conditions on it's non-vanishing in the following Proposition.

\begin{prop}\label{smallk}
Under the hypotheses of \propref{poincarenon-vanish} there exists an integer $C(m)$, depending on $m$ such that the Poincar\'{e} series $P_{n,r;0}^{k,m}$ doesnot vanish identically \[ \forall  k \in \left[ \max{ \left\{ C(m),\frac{g+7}{2}\right\}} , \infty \right) \mbox{ and } (n,r) \in \mbb{N} \times \mbb{Z}^{g}  \mbox{ with }  D \leq \frac{k^{\p}}{ \pi e} \cdot \det{(2m)} \] 
\end{prop}

\begin{rmk}
Though \propref{poincarenon-vanish} is applicable here, the above theorem accounts for (possibly) smaller values of $k$.
\end{rmk}

\begin{proof}  This theorem again follows from the arguments of the proof of \propref{poincarenon-vanish}. Here we use the following estimate for Kloosterman sums of degree $g$ (see \cite[p.508,512]{kohnen}):
\[ H_{m,c}(n,r,n\pm r) \leq 2^{\omega(c)} c^{-1/2 } (D,c) , \, \forall c \geq C(m) \]
where $C(m)$ is a constant depending on $m$.  
With the notation of \propref{poincarenon-vanish}, we have for some positive constant $C_{1}(m)$,
\begin{eqnarray*}
\begin{split}
|S(n,r)|  & \leq \underset{1\leq c \leq C(m)} \sum \, \frac{2^{\omega(c)+1} c^{g/2 -1} (D,c)}{\Gamma(k^{\p} +1)}\left(\frac{S}{2c}\right)^{k^{\p}} + \underset{c > C(m)} \sum \, \frac{ 2^{\omega(c)+1} c^{-1/2} (D,c)}{\Gamma(k^{\p} +1)}\left(\frac{S}{2c}\right)^{k^{\p}} \\  
& \leq C(m)^{(g-1)/2} \underset{1\leq c \leq C(m)} \sum \, \frac{2^{\omega(c)+1} c^{-1/2} (D,c)}{\Gamma(k^{\p} +1)}\left(\frac{S}{2c}\right)^{k^{\p}} + \underset{c > C(m)} \sum \, \frac{ 2^{\omega(c)+1} c^{-1/2} (D,c)}{\Gamma(k^{\p} +1)}\left(\frac{S}{2c}\right)^{k^{\p}} \\
& \leq \frac{2 C_{1}(m)}{\Gamma(k^{\p} +1)} \left(\frac{S}{2} \right)^{k^{\p}} \underset{c} \sum \frac{2^{\omega(c)}}{c^{k-(g+3)/2}}.\\
\end{split}
\end{eqnarray*}
The condition $k > \frac{g+7}{2}$ precisely guarantees convergence of the series above. The rest of the proof is identical to that of \propref{poincarenon-vanish}.
\end{proof}

\subsection{} We now come to the main result of this section.
\begin{proof}[\x{Proof of \thmref{order1+}} ]
We use Rankin's method as in \cite{rankin}. With $S(n,r)$ as above, we need to prove $|S(n,r)| < \frac{1}{2\pi}$. Define \[ \sigma = {k^{\p}}^{-1/6} , \q Q^{*} = \frac{2  \pi D}{ k^{\p} \det{( 2m)}}, \q  M(D) = \exp \left\{ \frac{ B_{1} \log{D}}{\log{ \log{2D}}} \right\}    \]

Define $H_{m,c}^{\pm}(n,r,n,r) = H_{m,c}(n,r,n,r)+ (-1)^{k} H_{m,c}(n,r,n,-r)$. Then we have $|S(n,r)| \leq    \det{(2m)}^{-1/2} |S_{1}(n,r)| +   \det{(2m)}^{-1/2} |S_{2}(n,r)| $, where
\begin{align*}
 & |S_{1}(n,r)| = \underset{1 \leq c \leq Q^{*}}\sum | H_{m,c}^{\pm}(n,r,n,r)| | J_{k^{\p}}\left( \frac{k^{\p} Q^{*}}{c} \right)| , \\
& |S_{2}(n,r)| = \underset{c \geq Q^{*}}\sum | H_{m,c}^{\pm}(n,r,n,r) | |J_{k^{\p}} \left( \frac{k^{\p} Q^{*}}{c} \right)| 
\end{align*}

We get after similar calculations as in \cite{rankin} (see also \cite{mozzochi}) that
\begin{align} 
|S_{1}(n,r)| & \leq A_{1} M(D) {Q^{*}}^{g/2 -1} \underset{d \mid D, d < Q^{*}}\sum 2^{\omega(d)} d \left\{ \frac{Q^{*} {\sigma}^{3}}{d} + 3 {\sigma}^{2}  \right\}  \nonumber \\  
& \leq A_{2}   M(D)^{3}  \frac{{Q^{*}}^{g/2 }}{{k^{\p}}^{1/2}} + A_{3} M(D)^{3} \frac{{Q^{*}}^{g/2 }}{{k^{\p}}^{1/3}} \nonumber \\
& \leq A_{4} M(D)^{3} \frac{ \left( \frac{ \pi D} {\det{(2m)}} \right)^{g/2 } } { {k^{\p}}^{g/2 + 1/2} } + A_{5} M(D)^{3} \frac{ \left( \frac{ \pi D} {\det{(2m)}} \right)^{g/2 } } { {k^{\p}}^{g/2 + 1/3} } \label{s1}
\end{align}

However, the other sum $S_{2}(n,r)$ needs to be handled differently. We have

\begin{align}
|S_{2}(n,r)| & \leq  \underset{Q^{*} < c \leq k^{\p} Q^{*}}\sum 2^{\omega(c)+1} c^{g/2 -1} (D,c) | J_{k^{\p}}\left( \frac{ 2 \pi D} { c \cdot \det{(2m)}} \right)|  \\
& +  \underset{ c > k^{\p} Q^{*}}\sum 2^{\omega(c)+1} c^{g/2 -1} (D,c) | J_{k^{\p}}\left(\frac{ 2 \pi D} { c \cdot \det{(2m)}}  \right)| \nonumber \\
& \leq 2 M(D)  \frac{  \left( \frac{ \pi D} { \det{(2m)}}  \right)^{k^{\p}} } {\Gamma{(k^{\p}+1)}} \underset{Q^{*} < c \leq k^{\p} Q^{*}}\sum \frac{1} {c^{k^{\p} -g/2} }  +  2 \underset{ c >  k^{\p} Q^{*}}\sum  c^{g/2 + 1} |J_{k^{\p}} \left( \frac{ 2 \pi D} { c \cdot \det{(2m)}}  \right)|   \nonumber \\
& \leq   \frac{ 2 M(D) } { {Q^{*}}^{k^{\p} - g/2 - 1 - \epsilon} }  \frac{  \left( \frac{ \pi D} { \det{(2m)}}  \right)^{k^{\p}} } {\Gamma{(k^{\p}+1)}} \underset{Q^{*} < c \leq k^{\p} Q^{*}}\sum \frac{1} {c^{1 + \epsilon} }   +  \,  \frac {2 \left( \frac{ \pi D} { \det{(2m)}}  \right)^{g/2 + 2 + \delta} }  {\Gamma{(k^{\p}+1)}}   \underset{c > k^{\p} Q^{*}}\sum \frac{1 }{c^{1 + \delta}}  \label{3.6}   \\
& \leq  \frac{ a_{0} M(D) } { {k^{\p}}^{ g/2 + 3/2 + \epsilon} }  \left( \frac{ \pi D} { \det{(2m)}}  \right)^{g/2 + 1 + \epsilon}   + \,  a_{1}  \frac{ \left( \frac{ \pi D} { \det{(2m)}}  \right)^{g/2 + 2 + \delta} }  { {k^{\p}}^{k^{\p} + 1/2} } \label{3.7} 
\end{align}
where $a_{i}, A_{j}$ are absolute constants, and $0 < \epsilon, \delta < 1$. Now for any $g \geq 1$, and $\alpha(g) = \frac{2}{3g+2}$; we choose $0 < \epsilon, \delta < \frac{1}{2}$ and find that $S_{1}(n,r)$ and $S_{2}(n,r)$ are small if we choose $k$ large. If $g \geq 5$, then we find that a better bound $\alpha(g) = \frac{2}{3g}$ works. This completes the proof.
\end{proof}

\section{Explicit basis for $J_{k,m}^{cusp}$ and proof of \thmref{g=1,indepofk}} \label{g=1}
\subsection{}
H. Petersson proved that the first $\ell = \dim{ S_{k} }$ Poincar\'{e} series $P^{k}_{1}, \cdots, P^{k}_{\ell}$ is a basis for the space of cusp forms $S_{k}$ for $SL(2, \mbb{Z})$. We prove the corresponding result for Jacobi forms. The proof is based on the dimension formula given in \cite{zagier}.

\begin{thm} \label{classicalbasis}
Let $k \geq m+12$. Then we have the following classical basis for $J_{k,m,1}^{cusp} \colon$ 
\begin{enumerate}

\item
If $k$ is even, $ \left \{ P^{k,m}_{D_{\mu}+ 4m \lambda_{\mu} , \mu} \right \}$ , $\mu = 0, 1, \ldots, m$ ; $\lambda_{\mu} = 0, 1, \ldots, \dim{ S_{k+2 \mu}} - \left [ \frac{\mu^{2}}{4m}\right ] - 1$  where $D_{\mu} := 4m\left( \left [ \frac{\mu^{2}}{4m}  \right ] +1 \right) - \mu^{2}$.

\item
If $k$ is odd, $ \left \{ P^{k,m}_{D_{\mu}+ 4m \lambda_{\mu} , \mu} \right \}$ , $\mu =  1, \ldots, m-1$ ; $\lambda_{\mu} = 0, 1, \ldots, \dim{ S_{k+2 \mu -1} } - \left [ \frac{\mu^{2}}{4m}\right ] - 1$ . 
\end{enumerate}
\end{thm}

\begin{proof}
We prove the Theorem for $k$ even, the other case is analogous. The condition $k \geq m+12$ ensures $\dim{ S_{k+2 \mu}} \geq \left [ \frac{\mu^{2}}{4m} \right] +1$ (see \cite[p.103]{zagier}). The proof follows Petersson's argument in the elliptic case (see \cite{petersson}, \cite{smart}). Let $d = \dim{J_{k,m}^{cusp}}$ and $\phi_{1}, \ldots, \phi_{d}$ be a orthonormal basis. We write 
\begin{align*}
& \phi_{j}(\tau,z) = \underset{\underset{D^{\p}>0, D^{\p} \equiv -r^{2} \pmod{4m}} {r \in \mbb{Z}} }\sum c_{j}(D^{\p},r) e \left( \frac{D^{\p}+ r^{2}}{4m} \tau + rz \right) \mbox{ and } \\
& P^{k,m}_{ D_{\mu} +4m \lambda_{\mu},\mu} = \lambda_{ k,m,D_{\mu} +4m \lambda_{\mu} }^{-1} \, \sum_{j=1}^{d} c_{j}(D_{\mu }+4m \lambda_{\mu},\mu) \phi_{j}. 
\end{align*}
where $\mu$ and $\lambda_{\mu}$ varies as in the statement of the Theorem. We get a $d \times d$ matrix indexed by pairs $(D_{\mu}+4m \lambda_{\mu} ,\lambda_{\mu})$ and $j$. It suffices to prove the matrix is invertible. If not, let there be a linear relation \[ \sum_{j=1}^{d}  \xi_{j} \, c_{j}(D_{\mu }+4m \lambda_{\mu},\mu) = 0, \, (\xi_{1}, \ldots, \xi_{d}) \neq (0, \ldots,0), \, \mbox{for all}\, (D_{\mu}+4m \lambda_{\mu} ,\mu). \]

\texttt{Claim :} Considering the non-zero Jacobi form $\Phi := \sum_{j=1}^{d}  \xi_{j} \phi_{j}$, we see that the Fourier coefficients $c_{\Phi} (D_{\mu }+4m \lambda_{\mu},\mu)$ ($\mu$ and $\lambda_{\mu}$ as in the Theorem) are zero. This implies that $D_{2\mu} \Phi = 0$ for $\mu = 0, \ldots, m$, (see \cite[p.32]{zagier} for the definition of operators $D_{2\mu}$) which shows that $\Phi = 0$, a contradiction.

\texttt{Proof of claim :} Let $\Phi \in J_{k,m}^{cusp}$. Then we have the following Fourier expansion of the modular form $D_{2 \nu} \Phi$ of weight $k + 2 \nu$, (cf. \cite[p. 32]{zagier}, $k \mbox{ even } \, , \nu = 0, \ldots, m$) 
\begin{equation} 
D_{2 \nu} \Phi = A_{k,\nu} \underset{n \geq 0}\sum \left(  \underset{r \colon r^{2} < 4mn}\sum   \left(  \underset{0 \leq \mu \leq \nu} \sum \frac{(k+ 2\nu - \mu -2)! }{(k + 2 \nu -2) ! } \frac{ (-mn)^{\mu} \, r^{2 \nu - 2\mu}} { \mu! \, (2 \nu - 2 \mu)!}    \right) c_{\Phi}(n,r) \right) q^{n} \label{Dfourier} 
\end{equation}
where, $A_{k, \nu} := (2 \pi i)^{-  \nu} \frac{(k+2 \nu -2)! \, (2 \nu)!} { (k+ \nu -2)!}$ and $q := e(\tau)$.

Let $\ell$ be an a positive even integer. Let $d_{\ell} := \dim{ S_{\ell}}$. Since an elliptic cusp form $f = \sum_{n=1}^{\infty} a(n,f) q^{n}$ of weight $\ell$ is determined by the first $d_{\ell}$ of it's Fourier coefficients $a(1,f), \ldots, a(d_{\ell},f)$, looking at equation~\ref{Dfourier} we need to prove that $c_{\Phi}(n_{\nu},r_{\nu}) = 0$ for all $r_{\nu}$ such that $r_{\nu}^{2} < 4mn_{\nu} \,, 0 \leq r_{\nu} \leq m $ and all $n_{\nu}$ such that $ [ \frac{r_{\nu}^{2}}{4m}] + 1 \leq n_{\nu} \leq d_{k+ 2 \nu}$ ($ \nu = 0, \ldots, m$). From now on let $\ell$ denote one of $k + 2 \nu$, ($\nu = 0, \ldots, m$) and for convenience, we drop the suffix $\nu$. To see this, first, if $|r| > 2m$ in equation~\ref{Dfourier}, we can consider $-m \leq r^{\p} = r - 2m x \leq m $ for a suitable integer $x$ and an $n^{\p} \geq 1$ such that $4m n^{\p} - {r^{\p}}^{2} = 4 m n - r^{2}$ and use the fact that $c_{\Phi}(n^{\p},r^{\p}) = c_{\Phi}(n,r)$ and that $n \geq n^{\p} \geq 1$, so $n^{\p}$ also satisfies the same upper bound as that of $n$ (namely, $ [ \frac{{r^{\p}}^{2}}{4m}] + 1 \leq n^{\p} \leq d_{\ell}$). We can finally reduce to the case $0 \leq r \leq m$ since $c_{\Phi}(n,r) = c_{\Phi}(n,-r)$ as $\ell$ is even. 

But any such $n_{\nu}$ can be written as $n_{\nu} = [ \frac{\nu^{2}}{4m}] + 1 + \lambda_{\nu} = \frac{ D_{\nu} + 4m \lambda_{\nu} + \nu^{2} }{4m}$ with $0 \leq \nu \leq m$ and $D_{\nu}, \, \lambda_{\nu}$ as in the statement of the theorem. This proves the claim.
\end{proof}

\subsection{}
The Eichler-Zagier map for Jacobi forms of integral weight and index $1$ denoted by $Z_{1} \colon J_{k,1} \rightarrow M_{k-1/2}^{+}$ is defined by \[ Z_{1} \colon \underset{ \underset {D \equiv - r^{2} \pmod{4} }{D>0, r \in \mbb{Z}} }\sum c(D) e \left( \frac{D+r^{2}}{4} \tau + rz  \right) \mapsto \underset{0<D \in \mbb{Z}}\sum  c(D)    e (D \tau) \]
where the Fourier coefficient $c(D)$ does not depend on $r$.

Let $k$ be even. Following the notation in \cite{kohnen1}, let $P_{k-1,4,D}$ ($(-1)^{k-1}D \equiv 0,1 \pmod{4}$) be the Poincar\'{e} series in $M_{k-1/2}^{+}$. The Fourier development of $P_{k-1,4,D}$ is given below:

\begin{prop}[Kohnen, \cite{kohnen1}]
\begin{equation*}
P_{k-1,4,D} (\tau) = \underset{ t \geq 1, (-1)^{k-1}n \equiv 0,1 (4) } \sum g_{D}(t) e(t \tau), \q \mbox{with}
\end{equation*}
\begin{align}
g_{D} (t) = \frac{2}{3} \left[ \delta_{D,t} + (-1)^{k/2} \pi \sqrt{2} (t/D)^{k/2 - 3/4} \underset{ c \geq 1}\sum H_{c}(t,D) J_{k - 3/2}\left(\frac{\pi}{c} \sqrt{tD} \right)      \right] .
\end{align}
Here $\delta_{t,D}$ is the Kronecker delta, and,
\begin{align}
H_{c}(t,D) = (1 - (-1)^{k-1}i) \left( 1 + \left( \frac{4}{c} \right) \right) \frac{1}{4c} \underset{ \delta(4c)^{*} }\sum  \left( \frac{4c}{\delta} \right)  \left( \frac{-4}{\delta} \right)^{k - 1/2} e_{4c} ( t \delta + D \delta^{-1} )
\end{align}
\end{prop}

\begin{defi}
(i) Let $w$ be an integer, $c \geq 1$ be even and $u,v \equiv 0,(-1)^{w} \pmod{4}$. Let $\alpha \in \{ 1,2 \}$. We define the following exponential sum,
\begin{align} \mathbb{H}_{\alpha c}(u,v) := (1 - (-1)^{w}i)  \frac{1}{4c} \underset{ 1 \leq \delta \leq \alpha c-1, \, (\delta, 4c) =1 }\sum  \left( \frac{4c}{\delta} \right)  \left( \frac{-4}{\delta} \right)^{w + 1/2} e_{4c} ( u \delta + v \delta^{-1} ),
\end{align} where $\delta \delta^{-1} \equiv 1 \pmod{4c}$.
\end{defi}

\begin{lem} \label{pairity}
 Let $w$ be an integer, $c \geq 1$ be even and $u,v \equiv 0,(-1)^{w} \pmod{4}$. \\
1. $H_{c}(u,v) =  (1 + (-1)^{u + v + c/2} ) \, \mathbb{H}_{2 c}(u,v) $ 
Therefore, when $c \equiv 2 \pmod{4}$, $H_{c}(u,v)$ vanishes unless $u,v$ have different pairity. \\
2. Let $c \equiv 0 \pmod{4}$. Then $H_{c}(u,v) =   (1 + (-1)^{u + v } ) (1 + e_{4}(u - v) ) \, \mathbb{H}_{c}(u,v) $. \\
\end{lem}

\begin{proof}
The proofs are obtained by splitting the exponential sums into half of the modulus in the sums. 
$1.$ This is easily seen by splitting the exponential sum as follows:
\begin{align}
& \underset{ 1 \leq \delta \leq 4c -1 \, , (\delta,4c)=1 }\sum  \left( \frac{4c}{\delta} \right)  \left( \frac{-4}{\delta} \right)^{w + 1/2} e_{4c} ( u \delta + v \delta^{-1} ) \nonumber \\
& = \underset{ 1 \leq \delta_{1} \leq 2c -1 \, , (\delta_{1},4c)=1 }\sum  \left( \frac{4c}{\delta_{1}} \right)  \left( \frac{-4}{\delta_{1}} \right)^{w + 1/2} e_{4c} ( u \delta_{1} + v {\delta_{1}}^{-1} )  \nonumber \\
& + \underset{ 1 \leq \delta_{1} \leq 2c -1 \, , (\delta_{1},4c)=1 }\sum  \left( \frac{4c}{\delta_{1} + 2c} \right)  \left( \frac{-4}{\delta_{1} + 2c} \right)^{w + 1/2} e_{4c} \left( u (\delta_{1}+ 2c) + v ( {\delta_{1}}^{-1} + 2c ) \right)  \nonumber \\
& = (1 + (-1)^{u + v + c/2} ) \, \underset{ 1 \leq \delta_{1} \leq 2c -1 \, , (\delta_{1},4c)=1 }\sum  \left( \frac{4c}{\delta_{1}} \right)  \left( \frac{-4}{\delta_{1}} \right)^{w + 1/2} e_{4c} ( u \delta_{1} + v {\delta_{1}}^{-1} ) 
\end{align}
This gives $1$. We omit the proof of $2$, as it is on the same lines as that of $1$.
\end{proof}

\begin{defi} Let $u,v,w \in \mbb{Z}$, $c \geq 1$. Define
\begin{align}
(i) \q & {\mathscr{H}}^{\p}_{c}(u,v) := \frac{1}{c} \, \left( \frac{4}{-c } \right) \, \left( \frac{-4}{c } \right)^{- w - 1/2} \, \underset{\delta(c)^{*}}\sum \left(  \frac{\delta}{c} \right)  e_{c} \left( u \delta + v \delta^{-1}  \right) \\
(ii) \q & \mathscr{H}_{4 c}(u,v) := \left( 1 + \left( \frac{4}{c} \right) \right) \frac{1}{4c} \underset{ \delta(4c)^{*} }\sum  \left( \frac{4c}{\delta} \right)  \left( \frac{-4}{\delta} \right)^{w + 1/2} e_{4c} ( t \delta + D \delta^{-1} )
\end{align}
\end{defi}

\begin{rmk}
In applications we will always use \lemref{pairity} and the above definitions in the case $w = k-1$, with $k$ even. 
\end{rmk}

\begin{defi}
We denote by $G(a,b,c)$ the Gauss sum defined by \begin{align} G(a,b,c) := \underset{ n \pmod{c}} \sum e_{c}(an^{2} + bn ) \q  \text{where} \q a,b,c \in \mbb{Z}.
\end{align}
\end{defi}

\begin{prop}[\cite{berndt}] \label{gaussprops} We have the following: \\
$1$. $G(a,b,c_{1} c_{2}) = G(c_{2}a,b,c_{1}) \, G(c_{1}a,b,c_{2})$, where $(c_{1},c_{2}) = 1$. \\
$2$. Let $(a,c) = 1$.
\begin{equation} \label{gauss}
G(a,b,c) = \begin{cases} \epsilon_{c} \sqrt{c} \left(  \frac{a}{c} \right) e_{c} \left( - \psi(a) b^{2}  \right) & \text{if~} c \equiv 1 \pmod{2} \, , 4a \psi(a) \equiv 1 \pmod{c}  \\
2 \, G(2a, b, \frac{c}{2}) & \text{if~} c  \equiv 2 \pmod{4} \, , b  \equiv 1 \pmod{2} \\
0 & \text{if~} c \equiv 2 \pmod{4} \, , b = 0 \\
(1+i) {\epsilon_{a}}^{-1} \sqrt{c} \left(  \frac{c}{a} \right) & \text{if~} c \equiv 0 \pmod{4} \, , b = 0 \\
0 &  \text{if~} c \equiv 0 \pmod{4} \, , b \equiv 1 \pmod{2}.            
             \end{cases}
\end{equation}
\end{prop}

\begin{prop} \label{kloorelns}
Let $c \geq 1$ and $k$ even. Then \, $H_{1,c}(n,r,n^{\p},\pm r^{\p}) =  H_{c}(D^{\p},D)$. 
\end{prop}

\begin{proof} 
We distinguish 3 cases, for classes of $c$ modulo $4$. Let $\epsilon_{\delta} = \begin{small} \begin{cases} 1 & \text{if~} \delta \equiv 1 \pmod{4} \, , \\
 i & \text{if~} \delta \equiv 3 \pmod{4}.                          
                           \end{cases} \end{small} $

$1$. $ \mathit{c \equiv 1 \pmod{2} }$ \\ We use the values of Gauss sums from table~(\ref{gauss}) in \propref{gaussprops}. 
\begin{align}
H_{1,c}(n,r,n^{\p},r^{\p}) & = c^{-3/2 } \, \underset{y(c)^{*}}\sum  G(\bar{y}, r \bar{y} + r^{\p}, c) \, e_{c}( n \bar{y} + n^{\p}y ) e_{2 c}(r r^{\p}) \label{12} \\
& = c^{-3/2 } \epsilon_{c} \sqrt{c} \, \underset{y(c)^{*}}\sum \left(  \frac{\bar{y}}{c} \right) e_{c} \left( - 4^{-1} y (r \bar{y} + r^{\p})^{2}  n + \bar{y} + n^{\p} y  \right)  e_{2 c}(r r^{\p})   \nonumber \\
& = \frac{ \epsilon_{c}}{c} \, \underset{y(c)^{*}}\sum \left(  \frac{\bar{y}}{c} \right)  e_{c} \left( D^{\p} y + 4^{-2} \bar{y} D  \right) e_{2}(D D^{\p} ) \nonumber \\
& = 2 (1+i) \mathscr{H}_{4c}(D^{\p}, 4^{-2} D) = 2 H_{c}(D^{\p},D), \, \text{after simplification}.
\end{align}
where $4 4^{-1} \equiv 1 \pmod{c}$ and the equality in the last line follows from \cite[p. 256, equation~(37)]{kohnen1}.

$2$. $\mathit{ c \equiv 2 \pmod{4}}$ \\ Let $c = 2 c^{\p}$, with $c^{\p}$ odd. From table~(\ref{gauss}), and \lemref{pairity}, we see that $H_{1,c}(n,r,n^{\p},r^{\p}) = 0 = H_{c}(D^{\p}, D)$ if $r$ and $r^{\p}$ or equivalently $D$ and $D^{\p}$ have the same pairity. When they have opposite pairity, using the multiplicative property of Gauss sum in \propref{gaussprops} and applying the formula from \ref{gauss} we have:
\begin{align}
H_{1,c}(n,r,n^{\p},r^{\p}) & = 2 c^{-3/2 } \underset{y(2 c^{\p})^{*}}\sum  G(2 \bar{y}, r \bar{y} + r^{\p}, c^{\p}) \, e_{2 c^{\p}}( n \bar{y} + n^{\p}y ) e_{4 c^{\p}}(r r^{\p}) \nonumber \\
&= 2 c^{-3/2 } \epsilon_{c^{\p}} \sqrt{c^{\p}} \, \underset{y(2 c^{\p})^{*}}\sum \left(  \frac{2 \bar{y}}{c^{\p}} \right) e_{c^{\p}} \left( - 8^{-1} y (r \bar{y} + r^{\p})^{2}  \right) e_{2 c^{\p}} \left( n \bar{y} + n^{\p} y \right) e_{4 c^{\p}}(r r^{\p})  \nonumber 
\end{align}
We make a change of variables $y \mapsto 2 y + c^{\p}$ and find after simplification that the above sum, (in which $y$ now varies over a reduced residue system modulo $c^{\p}$ and $ (2 y + c^{\p}) (2 \cdot 4^{-1} \bar{y} + c^{\p} ) \equiv 1 \pmod{2 c^{\p} }$, $4 4^{-1} \equiv 1 \pmod{c^{\p} }$) :
\begin{align}
&=  \frac{ \epsilon_{c^{\p}} }{\sqrt{2} c^{\p} } \underset{y(c^{\p})^{*}}\sum \left(  \frac{\bar{y}}{c^{\p}} \right)  e_{c^{\p}} \left( D^{\p} y + 4^{-3} \bar{y} D  \right) e_{2}(n+n^{\p} ) \nonumber \\
& = (1+i) \mathscr{H}_{8}(D^{\p},D) \, \mathscr{H}^{\p}_{c^{\p}}(D^{\p},4^{-3}D) = (1+i) \mathscr{H}_{8c^{\p}}(D^{\p},D) = H_{c}(D^{\p},D).
\end{align}
where the equalities in the last line follows from \cite[p. 256, equation~(38)]{kohnen1}.

$3$. $ \mathit{ c \equiv 0 \pmod{4}}$ \\ From table~(\ref{gauss}), and \lemref{pairity}, we see that $H_{1,c}(n,r,n^{\p},r^{\p}) = 0 = H_{c}(D^{\p}, D)$ if $r$ and $r^{\p}$ or equivalently $D$ and $D^{\p}$ have opposite pairity. When they have the same pairity, again applying the formula from \ref{gauss}, we have the following:
\begin{align}
H_{1,c}(n,r,n^{\p},r^{\p}) & =  c^{-3/2 } \, \underset{y(c)^{*}}\sum  G(\bar{y}, r \bar{y} + r^{\p}, c) \, e_{c}( n \bar{y} + n^{\p}y ) e_{2 c}(r r^{\p}) \nonumber \\
& = c^{-3/2 } \, \underset{y(c)^{*}}\sum G(\bar{y}, 0 , c) \, e_{4c}( D^{\p}y + D \bar{y}) \nonumber \\
& = (1 + i) c^{-3/2 } \, \underset{y(c)^{*}}\sum   {\epsilon_{\bar{y}}}^{-1} \sqrt{c} \left(  \frac{c}{\bar{y}} \right) \, e_{4c}( D^{\p}y + D \bar{y}) \nonumber \\
& = 4 \,  \mathbb{H}_{ c}(D^{\p},D)= H_{c}(D^{\p},D), 
\end{align}
where the equality in the last line follows from \lemref{pairity}($2$) with $w = k-1$ and the fact that $D \equiv D^{\p} \pmod{4}$.
\end{proof}

\begin{prop} \label{j2h}
$Z_{1}$ maps $P_{D,r} \in J_{k,1}^{cusp} $ to $ 3 \, P_{k-1,4,D} \in M_{k-1/2}^{+}$.
\end{prop}

\begin{proof}
First trivially we have, $\delta_{1}(n,r,n^{\p},\pm r^{\p})  = \delta_{D,D^{\p} }$. Therefore comparing the two Fourier developments and noting that $k$ is even, we see that it is sufficient to prove for all $c \geq 1$ that $H_{1,c}(n,r,n^{\p},\pm r^{\p}) = const. \cdot H_{c}(D^{\p},D)$.
Combining $1$, $2$ and $3$ from \propref{kloorelns}, we finally arrive at the conclusion that when $k$ is even, 
\begin{align}
c_{n,r}(n^{\p},r^{\p}) = 3 \, g_{D}(D^{\p}) \, \mbox{ for all } n,r,n^{\p},r^{\p},
\end{align}
where $c_{n,r}(n^{\p},r^{\p})$ and $g_{D}(D^{\p})$ are the coefficients on the Fourier expansions of the relevant Poincar\'{e} series defined above.
This completes the proof of \propref{j2h}.
\end{proof}

\begin{prop}
There exist positive constants $k_{0}$ and $B$, where $ B > 4 \, \log 2$, such that, for all even $k \geq k_{0}$ and all positive integers $D \leq k^{2}exp \{ -B \log k/\log \log k\}$, the Poincar\'{e} series $P_{k-1,4,D}$ and hence also the Poincar\'{e} series $P_{D,r}^{k,1}$ does not vanish identically. 
\end{prop}

\begin{proof}
From the Fourier expansion of $P_{k-1,4,D}$ given in \cite{kohnen1}, we see that the proof is the same as in the case of integral weight Poincar\'{e} series for congruence subgroups of $SL(2, \mbb{Z})$ given in \cite{mozzochi}; so we omit it.
\end{proof}

\begin{proof}[\x{Proof of \thmref{g=1,indepofk} }.] We write $S(n,r) = S_{1}(n,r) + S_{2}(n,r)$, where 
\begin{align*}
& S_{1}(n,r) =  i^{k} \pi \sqrt{2} m^{-1/2} \underset{1 \leq c \leq \frac{\pi D}{m} } \sum  H_{m,c}^{\pm}(n,r,n,r)  J_{k^{\p}} \left( \frac{\pi D}{m c} \right) \\
& S_{2}(n,r) = i^{k} \pi \sqrt{2} m^{-1/2} \underset{ c > \frac{\pi D}{m} } \sum H_{m,c}^{\pm}(n,r,n,r)  J_{k^{\p}} \left( \frac{\pi D}{m c} \right)
\end{align*}

We use the following estimate of Bessel functions to estimate $S_{1}(n,r)$: $ | J_{\nu}(r) | \leq A r^{-1/3}$, where $\nu \geq 0$, $r \geq 1$ (cf. \cite[Lemma 3.4]{bessel1}, the constant $C$ appearing in the Lemma can be computed to be the constant $A$ in \thmref{g=1,indepofk} using \cite[p. 333]{stein}.)  

\begin{align}
| S_{1}(n,r) |  &  \leq \frac{  2 \sqrt{2} \pi } { m^{1/2} }  \underset{1 \leq c \leq \frac{\pi D}{m} } \sum  \frac{2^{ \omega(c)} (D,c)}{c^{1/2}} | J_{k^{\p}} \left( \frac{\pi D}{m c} \right) | \nonumber \\
& \leq  \frac{ 2 \sqrt{2}  m^{1/3} \pi^{ 2/3} } { D^{1/3} m^{1/2} }  M \left( \frac{\pi D}{m} \right)  \underset{1 \leq c \leq \frac{\pi D}{m} } \sum \frac{(D,c)}{c^{1/6} } \nonumber \\
& \leq  \frac{2 \sqrt{2}  \pi^{ 2/3} } { D^{1/3} m^{1/6} }  M \left( \frac{\pi D}{m} \right) \underset{d \mid D, d < \frac{\pi D}{m}} \sum d \nonumber \\
& \leq  \frac {2 \sqrt{2}  D^{2/3}\pi^{ 5/3}  } { m^{7/6} }  M \left( \frac{\pi D}{m} \right) \sigma_{0}(D) \label{4.1}
\end{align}

\begin{align}
| S_{2}(n,r) |  & \leq  \frac{  2 \sqrt{2} \pi} { m^{1/2} }  \underset{ c > \frac{\pi D}{m} } \sum c^{3/2} | J_{k^{\p}} \left( \frac{\pi D}{m c} \right) | \nonumber \\
& \leq  \frac{  2 \sqrt{2} \pi} { \Gamma{(k^{\p} + 1)} m^{1/2} } \underset{ c > \frac{\pi D}{m} }\sum c^{3/2} \left( \frac{\pi D}{m c} \right)^{3/2 + 2 } \nonumber \\
& \leq  \frac{  2 \sqrt{2} \pi^{9/2} D^{7/2 }}  { \Gamma{(k^{\p} + 1)} m^{4  } } \underset{ c > \frac{\pi D}{m} }\sum \frac{1} { c^{2 } } \leq \frac{  2 \sqrt{2} \pi^{13/2} D^{7/2 } } { 6 \, \Gamma{(k^{\p} + 1)} m^{4 } } \label{4.2}
\end{align}

From the bound given in \thmref{g=1,indepofk}, it follows from estimates~(\ref{4.1}) and~(\ref{4.2}) that $S_{1}$ and $S_{2}$ are both less than $ \frac{1}{2 }$ in absolute value. Finally, from the expression of the $(n,r)$-\textit{th} Fourier coefficient of $P^{k,m}_{n,r}$ given in \propref{poincarefourier}, we get the Theorem.
\end{proof}

\section{Further results} \label{conditional}

Recall the one dimensional Kloosterman sum for a positive integer $c,$
\begin{equation}\label{1dim-kloosterman} 
S(r,m;c) = \underset{(h ,c) = 1}{\sum_{h=1}^{c}} e_{c}(rh + mh^{\p}), \mbox{ where } hh^{\p} \equiv 1 \pmod{c} 
\end{equation}
It is well known that (see \cite[\S 3]{rankin} for example) the following relation holds for a prime $p$ :
\begin{equation}\label{1dimidentity} 
S(rp^{\rho},mp^{\mu};c) = S(r,mp^{\rho+\mu};c) + p S(rp^{\rho-1},mp^{\mu-1};c/p), \mbox{ where } p|c,p \nmid r, p\nmid m  \, (\rho,\mu \geq 1).
\end{equation}

\begin{defi}
\mbox{ We let } \begin{align} K_{m,c}(n,r,n^{\p},r^{\p})  &= \underset{x(c),y(c)^{*}}\sum e_{c}\left(( m[x]+ rx +n)\bar{y} + n^{\p}y + r^{\p}x \right) \, (x \in \mbb{Z}^{g}/c \mbb{Z}^{g} , r \in \mbb{Z}^{g})\\ 
&= c^{g/2+1}\,e_{2c}\left(-r^{\p}m^{-1}r^{t}\right)H_{m,c}(n,r,n^{\p},r^{\p})
\end{align}
\end{defi}

\begin{lem} \label{higherkloosterman}Let $p$ be a odd prime such that \q $ p|(c,m,r,r^{\p}), \,  p \nmid n, p\nmid n^{\p}.$ Then the following identity holds :
\begin{align}\label{klooidentity}
K_{mp^{\mu},c}(p^{\mu}n,p^{\mu}r,p^{\rho}n^{\p},r^{\p}) = K_{mp^{\rho+\mu},c}(p^{\rho+\mu}n,p^{\rho+\mu}r,n^{\p},r^{\p}) + p^{2}K_{mp^{\mu-1},c/p}(p^{\mu-1}n,p^{\mu-1}r,p^{\rho-1}n^{\p},r^{\p}/p) 
\end{align}

\begin{proof}
The proof follows by noting that, 
\begin{equation}\label{1tohigherkloo}
K_{m,c}(n,r,n^{\p},r^{\p}) = \underset{x \pmod{c}}\sum e_{c}(r^{\p}x) \, S(n^{\p}, m[x]+rx+n; c) \, ,
\end{equation}
from which the $L.H.S.$ and the first term of the $R.H.S.$ in~(\ref{klooidentity}) are taken care of by summing both sides of the equation~(\ref{1dimidentity}) with appropriate arguments over $x \pmod{c}$. For the last term, we split the summation in~(\ref{1tohigherkloo}) (replacing $(m,n,r,n^{\p};c)$ by $(p^{\mu-1}m,p^{\mu-1}n,p^{\mu-1}r,p^{\rho-1}n^{\p};\frac{c}{p})$ respectively) as  $x = \frac{c}{p} x_{1} + x_{2}$ , where $x_{1}$ (resp.) $ x_{2}$ range over $\mbb{Z}^{g}/p \mbb{Z}^{g}$ (resp.) $ \mbb{Z}^{g}/\frac{c}{p}\mbb{Z}^{g}.$ We have 
\begin{eqnarray*}
\begin{split}
&  \underset{x \pmod{c}}\sum e_{c}(r^{\p}x) \, S\left(p^{\rho-1}n^{\p}, p^{\mu-1}(m[x]+rx+n); c/p \right)   \\
& = \underset{x_{1},x_{2}}\sum e_{c}\left(r^{\p}(c/p \,x_{1} + x_{2})\right) S\left( p^{\rho-1}n^{\p},p^{\mu-1}\left((c/p \,x_{1}+x_{2})^{t}m(c/p \,x_{1}+x_{2}) + r(c/p \,x_{1}+x_{2})+n\right);c/p \right)\\
& =  \underset{x_{1}}\sum e_{p}\left(r^{\p} c/p \, x_{1}\right) \underset{x_{2}}\sum e_{c/p}\left(r^{\p}/p \,x_{2}\right) S\left(p^{\rho-1}n^{\p}, p^{\mu-1}(m[x_{2}]+rx_{2}+n); c/p \right) \\
& = p K_{mp^{\mu-1},c/p}(p^{\mu-1}n,p^{\mu-1}r,p^{\rho-1}n^{\p},r^{\p}/p) ,\\ 
\end{split}
\end{eqnarray*}
Therefore using ~(\ref{1dimidentity}) the lemma follows.
\end{proof}
\end{lem}

\begin{proof}[\x{Proof of \thmref{consecutive}}]:  From \lemref{higherkloosterman} we easily deduce that under 
the conditions of the lemma, in particular, when $p \mid c$,
\begin{align}
H_{mp^{\mu},c}\left(p^{\mu}n,p^{\mu}r,p^{\rho}n^{\p},r^{\p}\right) & = H_{mp^{\rho+\mu},c}\left(p^{\rho+\mu}n,p^{\rho+\mu}r,n^{\p},r^{\p}\right) \nonumber\\
& + p^{-\frac{g}{2} +1} H_{mp^{\mu-1},\frac{c}{p}}\left(p^{\mu-1}n,p^{\mu-1}r,p^{\rho-1}n^{\p},\frac{r^{\p}}{p}\right). \label{457}
\end{align}
In the case $p \nmid c,$ we note that we have the equality from the definition,  
\begin{align} H_{mp^{\mu},c}\left(p^{\mu}n,p^{\mu}r,p^{\rho}n^{\p},r^{\p}\right) = H_{mp^{\rho+\mu},c}\left(p^{\rho+\mu}n,p^{\rho+\mu}r,n^{\p},r^{\p}\right) \label{458} \end{align}

We sum equation~(\ref{457}) over $c \geq 1$ such that $p \mid c$, equation~(\ref{458}) over all $c \geq 1$ and add them. Gathering all of above and noting that $\frac{2 \pi \sqrt{D^{\p}D}}{det(2m)\cdot c}$ is the same in all the three sums  \, (putting $\rho = \mu$ and $n^{\p} = n \,,r^{\p} = p^{\mu}r$),  we get positive constants $\alpha_{1}$ and $\alpha_{2}$ , such that 
\begin{equation*}
c^{k,p^{\mu}m }\left( p^{\mu}n,p^{\mu}r\right) = \alpha_{1}\,c^{k,p^{2\mu}m }\left( p^{2\mu}n,p^{2\mu}r ; n,p^{\mu}r\right) + \alpha_{2} \,c^{k,p^{\mu-1}m} \left( p^{\mu-1}n,p^{\mu-1}r \right)
\end{equation*}
(where we have used the notation $ c_{P^{k,m}_{n,r}}(n,r) := c^{k,m}(n,r;n,r) = c^{k,m}(n,r)$).
This immediately implies~(\ref{consecutivekloo}) and thus completes the proof of \thmref{consecutive}.
\end{proof}

\begin{rmk}
The constants $ \alpha_{1},\alpha_{2} $ in the above proof can be determined explicitly and may give a better result in the same vein as \thmref{smallk} (see \cite{rankin}).
\end{rmk}

\end{document}